\documentclass[12pt]{amsart}
\usepackage{amssymb,amsmath,amsthm}
\usepackage{mathptmx}

\def\e{{\rm e}}
\def\cic{\mathbf}
\def\eps{\varepsilon}

\def\d{{\rm d}}

\def\R {\mathbb{R}}
\def\N {\mathbb{N}}

\def\C {{\mathbb C}}
\def\C {{\mathbb C}}
\def\D {{\mathcal D}}

\def\Nn {{\mathcal{N}}}
\def\Cr{{\mathcal{W}}}

\def\F {{\mathcal F}}

\def\S{{\mathbf S}}
\def\size{{\mathrm{size}}}

\def\T{{\mathbf{T}}}

\def\M {{\mathsf M}}

\def\Ws {{\mathsf W}}
\def\Fs {{\mathsf F}}

\def \l {\langle}
\def \r {\rangle}

\def \and{\qquad\text{and}\qquad}

\def \no#1#2#3 {{\bf #1} (#3), #2.}
\def \eds#1#2#3 {#1, #2, #3.}
\newtheorem{proposition}{Proposition}
\newtheorem{theorem}{Theorem}

\newtheorem{lemma}[theorem]{Lemma}
\theoremstyle{definition}

\newtheorem*{remark}{Remark}

\numberwithin{theorem}{section}

\begin{document}
\title{Lacunary Fourier and Walsh-Fourier series  near $L^1$}
\author[F.\ Di Plinio]{Francesco Di Plinio}
\address{INdAM - Cofund Marie Curie Fellow at Dipartimento di Matematica, \newline \indent Universit\`a degli Studi di Roma ``Tor Vergata'', \newline  \indent Via della Ricerca Scientifica,   00133 Roma,  Italy   \newline \indent \centerline{and}   \indent
 The Institute for Scientific Computing and Applied Mathematics,
Indiana University
\newline\indent
831 East Third Street, Bloomington, Indiana  47405, U.S.A. }
\email{diplinio@mat.uniroma2.it   }
\subjclass{42B20}
 \keywords{Carleson operator, multi-frequency Calder\'on-Zygmund decomposition, pointwise convergence}
\thanks{The    author is an INdAM - Cofund Marie Curie Fellow and is  partially
supported by the National Science Foundation under the grant
   NSF-DMS-1206438, and by the Research Fund of Indiana University.}

\begin{abstract} We  prove the following theorem: given a lacunary sequence of integers $\{n_j\}$, the subsequences $\Fs_{n_j} f$ and $ \Ws_{n_j} f$ of respectively the Fourier and  the Walsh-Fourier series of $f: \mathbb T\to \C$ converge  almost everywhere to $f$ whenever
\begin{equation}
 \label{logintro} \tag{1}
\int_{{\mathbb T}} |f(x)| \log\log(\e^\e +|f(x)|)\log\log\log\log\big( \e^{\e^{\e^\e}}+ |f(x)|\big)   \d x <\infty.
\end{equation}
Our integrability condition \eqref{logintro} is less stringent than the homologous assumption in  the   almost everywhere convergence theorems of Lie \cite{LIE} (Fourier case) and Do-Lacey \cite{DL} (Walsh-Fourier case), where  a triple-log term appears in place of the quadruple-log term of \eqref{logintro}.  Our proof of the Walsh-Fourier case is  self-contained and, in antithesis to \cite{DL}, avoids the use of Antonov's lemma \cite{ANT,SS}, relying instead on the {novel} weak-$L^p $ bound  for the lacunary Walsh-Carleson operator
$$
\big\|\sup_{n_j} |\Ws_{n_j} f|\big\|_{p,\infty} \leq K \log( \e + p') \|f\|_p \qquad \forall 1<p\leq 2.
$$
\end{abstract}
\maketitle
\baselineskip=24truept
{\footnotesize  \noindent To appear on 
    \emph{Collectanea Mathematica}.    \vskip0.5mm \noindent
    Available \emph{Onlinefirst} at \verb|http://link.springer.com/article/10.1007/s13348-013-0094-3|
    \vskip0.1mm \noindent
    Received 14 Apr  2013, accepted August 2013.

\section{Introduction and main results} \label{s1}
Let $\mathbb{T}: \R\backslash \mathbb{Z} $ be the one dimensional torus, identified with the interval $[0,1)$, and write 
$$
\l f, g \r= \int_{\mathbb{T}} f(x) \overline{g(x)} \, \d x. 
$$
For each $f \in L^1({\mathbb T})$, one can construct the Fourier series of $f$
$$
 \Fs_n f(x)= \sum_{k=-n}^n \l f, E_k \r E_k (x), \qquad x \in \mathbb{T}, \; n \in \N
$$
where $E_k(x)= \e^{2\pi i k x}$,
as well as the
 Walsh-Fourier series of $f$
$$
  \Ws_n f(x) = \sum_{k=0}^n \l f, W_k \r W_k (x),  \qquad x \in \mathbb{T}, \; n \in \N
$$ 
where $\{W_{n}:n\in \mathbb{N}\}$ is the orthonormal basis of $L^2(\mathbb{T})$  defined as $$
 W_n(x)= \prod_{k\in \N} \big(\mathrm{sign}\sin(2^{k }  2\pi x) \big)^{\eps_k(n)}, \qquad \eps_k(n):=\lfloor 2^{-k} n \rfloor \, \mathrm{mod} \, 2.
$$
 We are interested in  almost-everywhere convergence of $\Fs_{n} f$, $\Ws_{n} f$ along lacunary subsequences of integers $\{n_j: j \in \N \}$, that is, sequences of integers for which
$$\inf_{j \in \N } \frac{n_{j+1}}{n_j} = \theta >1;
$$
the constant $\theta$ is termed the \emph{lacunarity constant} of  the sequence $\{n_j\}$.
 
   The first main result of this note is Theorem \ref{mainthm}  below. In the statement, as well as in the remainder of the paper, we adopt the notations
$$
\log_k(t)=\underbrace{\log\big(\cdots \big(\log\big(\e_{k}  +t\big)\cdots \big)}_{k \textrm{ times}}, \quad \e_0=1, \; \e_k:=\e^{\e_{k-1}}, \qquad k=0,1,2,\ldots
$$  
The precise definition of the Orlicz spaces $L\log_2L\log_b L(\mathbb T)$, $b=3,4$, appearing  in the statement of the theorem and in the subsequent discussion is postponed to the end of the introduction.
\begin{theorem}
\label{mainthm}
Let $\cic{n}=\{n_j: j \in \N \}$ be a $\theta$-lacunary sequence of integers.  The lacunary Carleson (resp.\ Walsh-Carleson) maximal operators 
$$\Fs^\star_{\cic{n}}f(x):= \sup_{n \in \cic{n}} |\Fs_{n}f(x)|, \qquad \Ws^\star_{\cic{n}}f(x):= \sup_{n \in \cic{n}} |\Ws_{n}f(x)|$$
 map  $L\log_2 L\log_4 L(\mathbb T) $ into $L^{1,\infty}(\mathbb T) $, with   operator norms depending only on $\theta$. As a consequence,  almost everywhere convergence of the lacunary partial sums
$$\Fs_{n_j} f(x) \to f(x), \qquad   \Ws_{n_j} f(x) \to f(x), \qquad \textrm{\emph{a.e.} } x \in \mathbb{T}$$   holds for all $f \in L\log_2L\log_4 L(\mathbb{T})$. \end{theorem} We send the interested reader to the survey article \cite{K2} and references therein for additional context and perspective on problems related to the almost-everywhere convergence of Walsh and of Fourier series (in particular, along lacunary subsequences). Here, we mention that Theorem \ref{mainthm} without the $\log_4$ term, which is the object of a conjecture by Konyagin  \cite{K2},  would be sharp  in the following sense: for any nondecreasing  $\phi:[0,\infty) \to [0,\infty) $ with $\phi(0)=0$ and $$  
 \phi(t)=o\big(t \log_2 (t)\big), \qquad t \to \infty,
$$ and any lacunary sequence $\{n_j\}$
there exists a function $f$ in $\phi(L)$ with lacunary  Fourier series divergent everywhere, i.e.\ 
$$
\int_{{\mathbb T}} \phi(|f(x)|) \, \d x < \infty \qquad \mathrm{and}
 \qquad  \sup_{j}|\Fs_{n_j}f(x)|= \infty \quad\forall x \in {\mathbb T} .
$$ This is due to Konyagin \cite{K1} as well; a perusal of the proof extends the construction to the Walsh-Fourier case.

The recent articles \cite{DL} and \cite{LIE} have made significant progress towards a positive solution of Konyagin's conjecture, respectively in the Walsh and in the Fourier setting. Their respective main results can be summarized as follows: given any lacunary sequence of integers $\{n_j\}$, the subsequence $\Ws_{n_j} f$ [resp.\ $\Fs_{n_j} f$]  converges almost everywhere to $f$ for all $f \in L\log_2 L \log_3 L(\mathbb{T})$.

The bulk of \cite{DL} is devoted to the proof of the following restricted weak-type estimate for the lacunary Walsh-Carleson   maximal operator: for all $\theta$-lacunary sequences $\cic{n}$,
\begin{equation}
 \label{rearrw}
 \|\Ws_{\cic{n}}^\star f\|_{1,\infty} \leq K  |F|  \log_2 \left( \frac{ 1}{|F|} \right), \qquad \forall |f| \leq \cic{1}_F, \, F \subset \mathbb{T},
\end{equation}
where $K$ is a positive constant depending only on the lacunarity constant $\theta$ of $\cic{n}$.
A subsequent application of Antonov's lemma \cite{ANT} improves  \eqref{rearrw} into the (modified) weak-type estimate\begin{equation}
\label{estimateww}
\|\Ws_{\mathbf{n}}^\star f\|_{1,\infty} \leq K  \|f\|_{1} \log_2 \left(\frac{\|f\|_\infty}{\|f\|_1}\right)
\end{equation}
for all bounded functions $f: \mathbb T \to \mathbb C$. 
In the later article   \cite{LIE},       a direct (that is, without first proving a restricted weak-type estimate and then achieving weak-type via Antonov's lemma) proof of  the Fourier analogue of \eqref{estimateww}, namely
\begin{equation}
\label{estimateff}
\|\Fs_{\mathbf{n}}^\star f\|_{1,\infty} \leq K  \|f\|_{1} \log_2 \left(\frac{\|f\|_\infty}{\|f\|_1}\right).
\end{equation} 
is given.
Once estimates \eqref{estimateww}-\eqref{estimateff} are in place, the bounds
\begin{equation} \label{bdtriv}
\Ws_{\mathbf{n}}^\star , \Fs_{\mathbf{n}}^\star : \mathfrak{W}\to L^{1,\infty}(\mathbb T), 
\end{equation}
the  space $\mathfrak{W}$ being the quasi-Banach rearrangement invariant space with quasinorm 
\begin{equation} \label{quasi}
\|f\|_{\mathfrak{W}} := \inf\left\{\sum_{k \in \mathbb \N}\log_1(k) \|f_k\|_1  \log_2 \left(\textstyle\frac{\|f_k\|_\infty}{\|f_k\|_1}\right) : \begin{array}{l} f= \textstyle \sum_{k \in \mathbb \N} f_k, \\  \textstyle \sum_{{k \in \mathbb \N}}  |f_k| < \infty \textrm{ a.e.}\end{array}\right\},
\end{equation}
follow, as described in \cite{LIE}, from an exploitation of Kalton's log-convexity of $L^{1,\infty} (\mathbb T)$ \cite{KAL}. A standard density argument then implies almost everywhere convergence of $\Ws_{\mathbf{n}}^\star f , \Fs_{\mathbf{n}}^\star f $ for functions $f \in \mathfrak{W}$.
 The space $\mathfrak{W}$ is akin to the $QA$ space of \cite{ADR}, and    the embedding  \begin{equation} \label{evstemb}L\log_2L\log_3L(\mathbb{T}) \hookrightarrow \mathfrak{W}\end{equation}  follows along the lines of the theory developed in \cite{ADR} for $QA$. In view of the above discussion, coupling the   embedding \eqref{evstemb}   with \eqref{bdtriv}   immediately leads to the main results of respectively \cite{DL}  and  \cite{LIE}.\footnote{The authors of \cite{DL} employ a differently defined  quasi-Banach space, denoted $Q_D$, and derive $Q_D \to L^{1,\infty}(\mathbb T$) boundedness of $\Ws_\cic{n}^\star$  from \eqref{estimateww},  as well as  the embedding $L\log_2L\log_3L(\mathbb{T})  \hookrightarrow Q_D$, by appealing to the results of \cite{CM}, which generalize Arias de Reyna's work \cite{ADR}. However, it can be inferred from the discussion in \cite[Section 1]{CM} that the spaces $Q_D$ and $\mathfrak{W}$ coincide in this particular case.}
Our observation   is that, in fact, the strengthening of \eqref{evstemb}
\begin{equation} \label{embeddac}
L\log_2L\log_4L(\mathbb{T}) \hookrightarrow \mathfrak{W}
\end{equation}     
also holds; hence, assuming  \eqref{bdtriv} again (e.\ g.\ in the Walsh case)  
\begin{equation}
\label{estimatelog4}
\|\Ws_{\mathbf{n}}^\star f\|_{1,\infty} \leq K  \|f\|_{L \log_2 L \log_4 L(\mathbb T)},
\end{equation}
which in turn implies  the almost everywhere convergence part of Theorem \ref{mainthm}.
 The elementary proof of  \eqref{embeddac}  is given in Section \ref{mainpf}. We claim no originality for the methods; similar arguments have appeared, for instance, in \cite{ADR,SS,GMS,CGMS}. 
 
  The second main goal of this article is to give a novel, self-contained proof   of the inequality \eqref{estimateww}, and hence of the Walsh-Fourier case of the bound \eqref{bdtriv}. Our proof is both  simpler, and richer, than the one of \cite{DL}: in particular, in antithesis to \cite{DL}, we bypass the intermediate step \eqref{rearrw}, thus avoiding the need for  Antonov's lemma.  Instead,  we recover  \eqref{estimateww} as an immediate consequence of  the  weak-type bound of Theorem \ref{thmp} below, which is of independent interest. \begin{theorem}\label{thmp} Let $\cic{n}=\{n_j\}$ be a $\theta$-lacunary sequence. There is a positive constant $K$, depending only on the lacunarity constant $\theta$ of $\cic{n}$, such that, for all $1<p\leq2$
$$\|\Ws_{\mathbf{n}}^\star f\|_{p,\infty} \leq K \log_1(p') \|f\|_p.$$ 
\end{theorem}  
  Note that weak and strong $L^p$ bounds for $\Ws_{\mathbf{n}}^\star$ with polynomial dependence on $p'$ of the operator norms follow by    standard (discrete) Littlewood-Paley theory; however, logarithmic dependence on $p'$ as in Theorem \ref{thmp} was previously unknown.
With this sharper estimate in hand, \eqref{estimateww}   is easily obtained via the chain of inequalities
\begin{align*}
\|\Ws_{\mathbf{n}}^\star f\|_{1,\infty}& \leq  \inf_{p >1}\|\Ws_{\mathbf{n}}^\star f\|_{p,\infty} \leq  \|\Ws_{\mathbf{n}}^\star f\|_{\bar p,\infty} \\ &\leq K \log_1(\bar p') \|  f\|_{\bar p }\leq K\log_1(\bar p')\|  f\|_{1 }\big(\textstyle \frac{\|f\|_\infty}{\|f\|_1}\big)^{\frac{1}{\bar p'}}
\end{align*}
 finally taking ${\bar p}'=\max\big\{2,\log\big(\frac{\|f\|_\infty}{\|f\|_1}\big)\big\}$.

A more detailed comparison of our approach to the proof in \cite{DL}, and a discussion on sharpness of Theorem \ref{thmp}, are given in the remarks Section \ref{sec5}.  Here, we just mention that one of the main tools of our proof  (appearing,  albeit in a different form, in \cite{DL} as well) is  a lacunary multifrequency Calder\'on-Zygmund decomposition argument (here, Lemma \ref{mflemma}), along the lines of \cite[Theorem 1.1]{NOT}. The structural\footnote{That is, modulo the usual technicalities due to the   spatial tails of the Fourier wave packets.} obstruction to  this scheme of proof when dealing with the Fourier case  is   that the  mean zero (with respect to multiple frequencies) part arising from the  multifrequency CZ decomposition, informally known as ``the bad part'', brings  nontrivial contribution, unlike the Walsh case. Despite the additional cancellation, we are unable to estimate this contribution efficiently as of now: overcoming these difficulties will be the object of future work.

A few words about notation. We will indicate by $\D$ the standard dyadic grid on $\R_+=[0,\infty)$ and by $\D_I=\{J \in \D: J \subseteq I\}$. Throughout, given a Young's function $\varphi$, we make use of the local Orlicz norms
$$
\|f\|_{L^{\varphi}(I)}:= \inf\Big\{\lambda>0 : \int_I \varphi\Big(\frac{|f(x)|}{\lambda}\Big) \, \frac{\d x}{|I|} \leq 1\Big\}, \qquad I \in \mathcal{D}.
$$
When $\varphi(t)=t^p$, $1\leq p \leq \infty$, we simply write  $L^p(I)$). With this notation, the usual $L^p$ Hardy-Littlewood dyadic maximal function is defined by 
$$
\M_p  f(x) = \sup_{ \D \ni I \ni x} \|f\|_{L^p(I)}. 
$$ 
With the notation $L\log_2L  \log_b L  (\mathbb{T})$, $b=3,4$ we indicate    the Orlicz (Banach) space defined by any  Young's function $\varphi_b$ with $t\log_2 (t) \log_b (t)  = \varphi_b(t)$ for $t>\e_b$. We observe for future use that $L\log_2L\log_b L(\mathbb{T})$ is a Banach space with unit ball 
$$B_b=
\Big\{f: \mathbb T \to \C,  |f|_{\mathcal{L}\log_2 \mathcal{L}\log_b\mathcal{L}}:=\textstyle\int_{\mathbb{T}} \varphi_b(|f (x)|) \, \d x \leq 1\Big\},
$$

Finally, the positive constants implied by the almost inequality signs appearing in the remainder of the paper are meant to be absolute unless otherwise specified: in that case, we will adopt the notation $\lesssim_a$ to indicate dependence of the implied constant on the parameter $a$. When we write $A \sim B$, we mean that $A \lesssim B$ and $B \lesssim A$ (and analogously for $\sim_a$).

The article is organized as follows. In the forthcoming Section \ref{mainpf}, we prove \eqref{embeddac}, which in turn implies Theorem \ref{mainthm}, via estimates \eqref{estimateww}, \eqref{estimateff}. In Section \ref{tf}, we review the discretization of the operator $\Ws_{\cic{n}}^\star$ into the model sum $C^{\cic{n}} $  and prove an auxiliary exponential estimate. This exponential estimate, together with a multi-frequency   projection argument exploiting the lacunary structure of the frequencies  (Lemma \ref{mflemma}),  
are the cornerstones of the proof of Theorem \ref{thmp}, given in Section \ref{s4}. Section \ref{sec5} contains additional remarks and open problems.  
\section{Proof of the embedding \eqref{embeddac}} \label{mainpf}   
To prove \eqref{embeddac}, in view of the definition \eqref{quasi} of the quasinorm on $\mathfrak{W}$,  it suffices to show that for any $f$ in the unit ball $B_4$ of $L\log_2L\log_4L (\mathbb{T})$
  there exists a sequence $\{f_k:k \in \mathbb \N\}$ with \begin{equation} \label{strim11} f=\sum_{k \in \mathbb \N}f_k, \quad \sum_{k \in \mathbb \N}|f_k| < \infty \;\textrm{ a.e.,}\quad    \sum_{k \in \mathbb \N}\log_1(k) \|f_k\|_1  \log_2 \left(\textstyle\frac{\|f_k\|_\infty}{\|f_k\|_1}\right) \lesssim 1.
\end{equation}
Given such an $f \in B_4$, we define $\{f_k:k \in \mathbb \N\}$ by   $$ f_k:= f \cic{1}_{F_k}, \qquad F_0 = \{|f|\leq \e^\e\}
,\; F_k=\Big\{\e^{\e^{\e^k}} < |f| \leq \e^{\e^{\e^{k+1}}}\Big\},\; k \geq 1. 
$$ The absolute convergence almost everywhere of the series is immediate, since each $f_k$ is bounded and the supports of the $|f_k|$ are pairwise disjoint. We use the elementary fact that $$
x \in F_k \implies \log_2(|f(x)|) \log_4(|f(x)|) \sim \e^{k} \log_1 k.$$ Consequently, adopting the shorthand $A_k:=|f_k|_{\mathcal{L}\log_2 \mathcal{L}\log_4\mathcal{L}}$
$$
\|f_k\|_1 \sim \frac{A_k}{\e^k \log_1(k)}, \qquad \frac{\|f_k\|_\infty}{\|f_k\|_1} \lesssim \frac{\e^{\e^{\e^{k+1}}}\e^k \log_1(k)}{A_k},
$$
whence
\begin{equation} \label{strim1}
\|f_k\|_1 \log_2\left( \frac{\|f_k\|_\infty}{\|f_k\|_1} \right) \lesssim \frac{A_k}{\e^k \log_1(k)} \log_2 \left( \frac{\e^{\e^{\e^{k+1}}}\e^k \log_1(k)}{A_k} \right). \end{equation}
We separate two regimes. In the regime $$
R_1=\left\{k: \frac{A_k}{\e^k \log_1(k)} \geq \frac{1}{\e^{\e^{\e^{k+1}}}} \right\},$$ the above inequality turns into
\begin{equation} \label{strim2}
\|f_k\|_1 \log_2\left( \frac{\|f_k\|_\infty}{\|f_k\|_1} \right)   \lesssim \frac{A_k}{\e^k \log_1(k)} \log_2 \left(  \big(\e^{\e^{\e^{k+1}}}\big)^2\right) \lesssim   \frac{A_k}{ \log_1(k)}. \end{equation} In the complementary regime $R_2$, using the trivial inequalities $$4\log_2(ab) \leq 2 \log_2(a) \log_2(b), \qquad \forall a,b>0$$ and $a \log \frac{1}{a} \leq \sqrt{a}$ for $|a| \leq 1$, 
 \eqref{strim1} becomes
\begin{align} \label{strim3}
\|f_k\|_1 \log_2\left( \frac{\|f_k\|_\infty}{\|f_k\|_1} \right) \ & \lesssim  \textstyle\frac{A_k}{\e^k \log_1(k)} \log_2 \left( \textstyle\frac{\e^k \log_1(k)}{A_k} \right) \log_2 \left(  {\e^{\e^{\e^{k+1}}}} \right) \nonumber \\ &\lesssim \left(\textstyle\frac{A_k}{\e^k \log_1(k)} \right)^{\frac12} \e^{k}  \lesssim \displaystyle \frac{\e^k}{ \e^{\e^{\e^{\frac k 2}}} } \lesssim \e^{-k}.    \end{align}
With  \eqref{strim2}-\eqref{strim3} in hand, we   easily get the last part of \eqref{strim11} as follows:
\begin{align*}
  \sum_{k \in \mathbb N} \log_1(k)  \|f_k\|_1 \log_2\left( \frac{\|f_k\|_\infty}{\|f_k\|_1} \right) &   \lesssim   \sum_{k \in R_1} A_k   + \sum_{k \in R_2} \e^{-k} \log_1(k)   \\ &\lesssim    |f |_{\mathcal{L}\log_2 \mathcal{L}\log_4\mathcal{L}} + 1 \lesssim 1.
\end{align*}
 The proof of the embedding \eqref{embeddac}  is thus completed. 
\section{Discretization and an exponential estimate} \label{tf}

 A \emph{bitile} $s=I_{s} \times \omega_{s} \in \D_{\mathbb T} \times \D $ is a dyadic rectangle with $|\omega_s|= 2|I_s|^{-1}$.
 We think of $s$ as the union of the two \emph{tiles} (dyadic rectangles in $\D_{\mathbb T} \times \D$ of area 1)
 $$
 s_1 = I_s \times \omega_{s_1}, \qquad  s_2 = I_s \times \omega_{s_2}
 $$ where  $\omega_{s_1},\omega_{s_2}$ refer respectively to the left and right dyadic children of $\omega_{s}$. The set of all bitiles will be denoted by $\mathbf{S}_{\mathbb{T}}$.
For each tile $t=I_t \times \omega_t$, the corresponding Walsh wave packet is defined by
$$
w_t(x) = \mathrm{Dil}^{2}_{|I_t|} \mathrm{Tr}_{\inf I_t} W_{n_t} (x)=|I_t|^{-1/2}W_{n_t}\Big(\frac{x-\inf I_t}{|I_t|}\Big), \qquad n_t:= |I_t| \inf \omega_t.
$$
Let $N: \mathbb{T} \to \R_+$ be a  measurable choice function and consider the model sum  
$$
C_{ {\S_{{\mathbb T}}}} f (x) =    \sum_{s \in\S_{{\mathbb T}}} \l f, w_{s_1}\r w_{s_1}(x) \cic{1}_{\omega_{s_2}} (N(x));
$$
we do not indicate the dependence on the choice function in our notation. This model sum is the discretization of the \emph{unrestricted} maximal operator $ \Ws^* f:=\sup_{n\in \N} | \Ws_{n } f|  .$ To obtain a faithful model sum for the maximal partial sum $\Ws^\star_{\cic{n}}$ restricted to the (lacunary) sequence $\cic{n}=\{n_j\}$, we restrict the range of the choice function $N$ to values in $\cic{n}$; this restricts the sum over the bitiles $\S_{\mathbb T}^{\cic{n}}:=\{s \in \S_{\mathbb{T}}:  \omega_{s_2} \cap \cic{n} \neq \emptyset \}$, whence  the equivalence \cite{ThWp}
$
\Ws^\star_{\cic{n}} f \sim C_{\S_{\mathbb T}^{\cic{n}}} f 
$. In the remainder of the article, we use the simpler notation $C^{\cic{n}}$ in place of $C_{\S_{\mathbb T}^{\cic{n}}}$  and further denote by
$$
C_\S   f (x) =    \sum_{s \in\S } \l f, w_{s_1}\r w_{s_1}(x) \cic{1}_{\omega_{s_2}} (N(x))
$$
the model sum corresponding to an arbitrary finite subcollection $\S \subset \S_{\mathbb T}^{\cic{n}} $.

 
The remainder of this section is devoted to the following  proposition, upon whose proof Theorem \ref{thmp} relies. 
\begin{proposition}\label{thmexp} Let $\cic{n}=\{n_j\}$ be a $\theta$-lacunary sequence.
 Then 
$$
\big|\{x \in \mathbb{T}: |C^{\cic{n}} f(x)|\gtrsim \lambda \}  \big| \lesssim_\theta \exp \left( -\frac{\lambda}{\|f\|_{\infty}} \right), \qquad \lambda >0,
$$
that is, $C^{\cic{n}} : L^\infty(\mathbb{T}) \to \mathrm{exp}(L^1) (\mathbb T).$
\end{proposition}
The proof of Proposition \ref{thmexp} is given in Subsection \ref{ss32}; in the forthcoming Subsection \ref{ss31}, we recall the necessary tools of time-frequency analysis. 
\subsection{Analysis and combinatorics in the Walsh phase plane} \label{ss31} The material of this subsection is   essentially lifted from earlier work \cite{DD2} (see also \cite{ThWp}), with the exception of Lemma \ref{fefftrick}, which exploits the lacunary structure of the frequencies.

We begin by  recalling the well-known Fefferman order relation on either tiles or bitiles
\begin{equation}
\label{feff}
s \ll s' \iff I_s \subset I_{s'}\text{ and } \omega_{s} \supset \omega_{s'}.
\end{equation} A collection $\S\subset \S_{\mathbb{T}}$ is called \emph{convex} if
\begin{equation}
s, s'' \in \S, \,s' \in \S_{\mathbb{T}},\, s  \ll s' \ll s'' \implies s' \in  \S.
\end{equation}
We will use below that the collection of convex subsets is closed under finite intersection.

Given a set of bitiles $\S$, let  $\Pi_{\S}$ denote the orthogonal projection on the subspace of $L^2(\mathbb{T})$ spanned by
$\{w_{s_j}: s \in \S, j=1,2\}$.
  We set, for $f \in L^2(\mathbb{T})$,
$$
\size_{f} (\S) = \sup_{s \in \S}   \frac{\| \Pi_{\{s\}}f  \|_2}{\sqrt{|I_s|}}.
 $$
Note that$$
\size_{f} (\S) \sim \sup_{s \in \S }\sup_{j=1,2}   \frac{|\l f, w_{s_j} \r |}{\sqrt{|I_s|}}.
 $$
so that
 \begin{equation}
\label{ubsize}
\size_f(\S) \leq \sup_{s\in \S} \inf_{x \in I_s} \M_1 f(x).
\end{equation}
A collection of bitiles  $\T\subset \S$  is called a \emph{tree} with top bitile $s_\T  $ if $s\ll s_\T$ for all $s \in \T$. We use the notation $I_\T:=I_{s_\T}, \omega_{\T}=\omega_{s_\T}$. To characterize the contribution region of a tree, it is useful to introduce the notion of  the \emph{crown} of a tree:
$$
\mathsf{cr}(\T)= \bigcup_{s \in \T} \omega_{s_2}.
$$
We have the following exponential-type estimate for the model sum restricted to a tree of definite size. Note that $C_{\T} f$ is supported on $I_\T$.
\begin{lemma} \label{treelemma}Let $\T$ be a convex tree and 
 $\sigma = \size_f(\T)$. Then
$$
\big|\{x \in I_\T: |C_{\T} f(x)|\gtrsim \lambda \sigma \}  \big| \lesssim \e^{-\lambda} |I_T|, \qquad \forall \lambda >0.
$$
\end{lemma}
\begin{proof} It is obvious that $C_\T f= C_\T \Pi_\T f$, hence the lemma follows from the bound $$
\|C_\T (\Pi_\T f)\|_{\mathrm{BMO}(\mathbb T)} \lesssim \|\Pi_\T f\|_{\infty} \leq \size_f(\T)
$$
 and the John-Nirenberg inequality. For details on the second inequality see (for instance) \cite{DD2}.
\end{proof}
A finite convex collection of bitiles $\S$ is called a \emph{forest}  if $\S$ can be partitioned into (pairwise disjoint) convex trees $\{\T: \T \in \F\}$. It may be that a given $\S$ may admit many such partitions $\F$. The \emph{counting} function and  the \emph{crown} function of the forest $\S$ with respect to the partition $\F$ are respectively defined as $$
\Nn_\F (x) = \sum_{\T \in \F } \cic{1}_{I_\T} (x), \qquad \Cr_\F(x)=\sum_{\T \in \F } \cic{1}_{I_\T} (x)\cic{1}_{\mathsf{cr}(\T)} (N(x)) 
$$
For a tree $\T$, $\mathrm{supp} C_\T f \subset I_\T \cap N^{-1}(\mathsf{cr}(\T))$, and as a consequence, for a forest $\S$ with partition $\F$, one has the pointwise inequality
\begin{equation}
\label{crownin}
|C_\S f(x)| \leq \Cr_{\F} (x) \max_{\T \in \F} |C_\T f(x)|.
\end{equation}
 The lemma below can be used to decompose any convex collection of bitiles into forests of definite size, keeping the  the $L^1$ norm of the counting functions under control. See \cite{DD2} for a proof.
 \begin{lemma} \label{sizelemma} Let $\S $ be a finite convex collection of bitiles  with $\size_f(\S) \leq A$.
We can decompose $\S= \bigcup\{\S_{\sigma} :  {\sigma \in 2^{-\mathbb N} } \}$, with each $\S_{\sigma }$  a forest  such that\begin{align} \label{sizelemma1} &\size_f(\S_{\sigma }) \leq  A\sigma, \\ \label{sizelemma2} &\|\Nn_{\F_{\sigma }}\|_{1} \lesssim \sigma^{-2} A^{-2}\|f\|_{2}^2, \end{align}
for some partition $\F_\sigma$.
 \end{lemma}
Our last lemma is specific of the lacunary case: in view of the fact that each bitile contains elements from the lacunary sequence $\cic{n}$, we have a bound on the crown function of a generic forest which only depends on the lacunarity constant $\theta$.
\begin{lemma}\label{fefftrick}For any  forest $\S \subset \S^{\cic{n}}_{\mathbb{T}}$ with partition $\F$, there is a partition $\F^\star$ with
$$
\|\Cr_{\F^\star}\|_\infty \lesssim_\theta 1, \qquad \|\Nn_{\F^\star }\|_{1} \lesssim_\theta \|\Nn_{\F}\|_{1}.
$$
\end{lemma}
\begin{proof} It suffices to show that $\S$ can be split into $\sim_\theta 1$ forests $\S^j$ with partitions $\F^j$, such that
 $$\|\Nn_{\F^j }\|_{1} \lesssim_\theta \|\Nn_{\F}\|_{1}, \qquad \{I_\T \times \mathsf{cr}(\T): \T\in \F^j  \} \,\textrm{ pairwise disjoint}.$$ We define $\S^0:= \{ s \in \S: n_1 \in \omega_{s}\}$. It is clear that $\S^0$ can be partitioned into convex trees $\T \in \F^0$ with pairwise disjoint $I_\T$ (take the $\ll$-maximal bitiles in $\S^0$ as tops). For each of these trees there exists a unique tree $\T'\in \F$ such that the top bitile $s_\T \in \T'$, whence $|I_\T|\leq |I_{\T'}|$; it then follows that $\|\Nn_{\F^0 }\|_{1} \leq\|\Nn_{\F}\|_{1}.$
Let now
  $\tilde\S= \S \backslash \S^0$ and ${\tilde\S}^\star$ be the $\ll$-maximal bitiles of $\tilde \S$. 
It should be apparent that $\sum_{s\in {\tilde\S}^\star } |I_s| \leq \|\Nn_\F\|_1$.
By the Fefferman trick (see for example Section 5 of \cite{D}), the initial claim will   follow if we show that for each $s \in \tilde \S$
$$
M:=\max_{s \in \tilde \S}   \# \big(T(s):=\{s'\in {\tilde\S}^\star: I_s \subset I_{s'}, \omega_{s'} \subset \omega_{s_2}   \}\big)    \lesssim_\theta 1
$$ 
Take $s \in \tilde \S$ which attains the maximum $M$. The collection  $T(s) $ is made of pairwise disjoint bitiles with $I_s \subset I_{s'}$, thus the intervals $\{\omega_{s}:  s \in T(s)  \}$ must be pairwise disjoint, and each contains a different $n_j \in \cic{n}$. It follows that $\omega_{s_2}$ contains at least $M$ different frequencies. Let $n_j$ and $n_k$ be the minimum and the maximum of these frequencies respectively. It must be $k\geq j+M$, whence $|\omega_{s_2}| \geq n_k-n_j  \geq (\theta^M-1)n_j $. If $M \geq \frac{\log 2}{\log \theta}$, we would have $|\omega_{s_1}| \geq   n_j$, $\inf \omega_{s_1} \leq n_j$, which in turn would imply  $n_1 \in \omega_{s}$, and $s$ would have been selected for $\S^0$. Thus $M \leq \frac{\log 2}{\log \theta} \lesssim_\theta 1 $ as claimed. 
\end{proof}


 \subsection{Proof of Proposition \ref{thmexp}} \label{ss32}
It suffices to  argue for $\lambda >\|f\|_\infty$ (the statement is otherwise trivial).   Furthermore, by a limiting argument, we may  argue for $C_\S$ in place of $C^{\cic{n}}$, with    $\S$ an arbitrary finite convex subcollection of $\S_{\mathbb T}^{\cic{n}}$, ensuring that the implied constants do not depend on $\S$.

    A consequence of   \eqref{ubsize} is that $\size_f(\S) \leq \|f\|_\infty$, and we can   apply the size decomposition Lemma \ref{sizelemma},  with $A=\|f\|_\infty$.
 We further apply Lemma \ref{fefftrick} to the resulting forests $\{\S_\sigma\}_{\sigma \in 2^{-\N}}$ with $\size_f(\S_\sigma) \leq \sigma\|f\|_\infty$, yielding partitions $\F_\sigma$ with 
\begin{equation} \label{thmexp1}
\|\Nn_{\F_\sigma}\|_1 \lesssim \sigma^{-2}\|f\|_\infty^{-2}\|f\|_2^2, \qquad \|\Cr_{\F_\sigma}\|_{\infty} \lesssim 1. 
\end{equation}
We will show that
\begin{equation} \label{excsetdef}
\{|C_\S f|\gtrsim \lambda \} \subset E:=\bigcup_{\sigma \in 2^{-\mathbb N}}   \bigcup_{\T \in \F_{\sigma}} E_\T,
\end{equation} 
where
$$
E_\T:= \big\{x \in I_T: |C_\T f|\gtrsim \lambda   \sigma \log \textstyle\big(\frac{1}{\sigma^{4}}\big) \big\}.
$$
Note that,  applying Lemma \ref{treelemma}, 
$$
|E_\T|=\big|\big\{x \in I_\T: |C_\T f|\gtrsim \textstyle \frac{\lambda}{\|f\|_\infty}   \ \log \textstyle\big(\frac{1}{\sigma^{4}}\big) \size_f(\T) \big\}\big| \lesssim \exp\big( \textstyle- \frac{\lambda}{\|f\|_\infty}\big)   \sigma^4    |I_\T| ,
$$
whence, in view of \eqref{thmexp1},
$$
|E|  \leq \exp\big({ \textstyle- \frac{\lambda}{\|f\|_\infty}}\big) \sum_{\sigma \in 2^{-\mathbb N}}\sigma^4 \|\Nn_{\F_{\sigma}}\|_{1} \lesssim \exp\big( \textstyle- \frac{\lambda}{\|f\|_\infty}\big) \|f\|_\infty^{-2}\|f\|_2^2.  $$
Therefore, assuming for a moment the inclusion   \eqref{excsetdef}, we have arrived at 
 \begin{equation} \label{forlater}
 \{|C_\S f|\gtrsim \lambda \} \lesssim \exp\big( \textstyle- \frac{\lambda}{\|f\|_\infty}\big) \|f\|_\infty^{-2}\|f\|_2^2;
\end{equation}
Proposition \ref{thmexp} simply follows from the obvious $\|f\|_\infty^{-1}\|f\|_2 \leq 1$.
The above mentioned inclusion is proved by observing that
$$
\sup_{x \in E^c} \sup_{\T \in \F_\sigma} |C_\T f(x)|\leq \lambda   \sigma \log \textstyle\big(\frac{1}{\sigma}\big),
$$
and therefore, making use of the triangle inequality, \eqref{crownin}, and \eqref{thmexp1},
\begin{align*}
|C_\S f(x)| &\leq \sum_{\sigma \in 2^{-\mathbb N}} |C_{\S_\sigma} f(x)| \leq \sum_{\sigma \in 2^{-\mathbb N}} \|\Cr_{\F_\sigma}\|_\infty \sup_{\T \in \F_\sigma} |C_\T f(x)| \\ &\lesssim_\theta  \lambda \sum_{\sigma \in 2^{-\mathbb N}} \sigma \log \textstyle\big(\frac{1}{\sigma} \big) \lesssim_\theta \lambda
\end{align*}
for $x \in E^c$,
which means that $E^c \subset \{C_\S f \lesssim \lambda\}.$ The proof of Proposition \ref{thmexp} is thus completed.
 
 \begin{remark} Perusing the proof of Proposition 
\ref{thmexp}, we realize that we have proved the following estimate: for a finite convex $\S \subset \S_{\mathbb{T}}^{\cic{n}}$, and any $A \geq \size_f(\S) $, \begin{equation} \label{remarkexpest}
\big|\{x \in \mathbb{T}: |C_{\S} f(x)|\gtrsim \lambda \}  \big| \lesssim_\theta \exp \left( \textstyle-\frac{\lambda}{A} \right)   \frac{\|f\|_2^2}{A^2},\qquad \lambda >0.
\end{equation}
\label{remarkexp}
This estimate will be used in the proof of Theorem \ref{thmp}.
\end{remark} 
 \section{Proof of Theorem \ref{thmp}} \label{s4} By the usual limiting argument,   replacing  $\S_{\mathbb{T}}^{\mathbf{n}} $ with an arbitrary finite convex subcollection $\S$,
Theorem \ref{thmp} is equivalent to the estimate
\begin{equation} \label{mainwte}
\big|\{ x \in \mathbb{T}:|C_{\S } f (x)| \gtrsim  \log_1(p') \lambda       \} \big| \lesssim_\theta \frac{  \|f\|^p_p}{\lambda^p}, \qquad \forall \lambda>0.
\end{equation} 
Furthermore, by scaling $f$, it suffices to work with $\lambda=1$.

First of all, note that the left-hand side of \eqref{mainwte}  is less than or equal to
\begin{equation} \label{s4pf1}
 \big|\{ x \in \mathbb{T}: \M_{p} f (x) >    1     \} \big| + \big|\{ x \in \mathbb{T}:|C_{\S } f (x)| \gtrsim  \log_1(p')  , \M_p   f(x) \leq    1   \} \big|
\end{equation}
and the first summand complies with the bound on the right-hand side of \eqref{mainwte} by the maximal theorem. Thus it suffices to estimate the second summand of \eqref{s4pf1}; note that
$$\M_p    f(x) \leq {1} \implies C_{\S } f (x) = C_{\S^{1} } f (x), \qquad \S^{1}=\left\{s \in \S: \inf_{I_s} \M_1 f \leq {1} \right\},
$$
and thus it suffices to estimate\begin{align} \label{s4pf2} 
\big|\{ x \in \mathbb{T}:|C_{{\S^{1}} } f (x)| \gtrsim   \log_1(p')      \} \big| &\leq \big|\{ x \in \mathbb{T}:|C_{{\S^{1}} } f_1 (x)| \gtrsim \log_1(p')      \} \big|  \nonumber \\&   +\big|\{ x \in \mathbb{T}:|C_{{\S^{1}} } f_2 (x)| \gtrsim \log_1(p')      \} \big|,
\end{align} 
where $f_1:= f \cic{1}_{\{\M_p f \leq {1}\}}$, $f_2:= f-f_1 $.
Our reduction has resulted into
\begin{equation}
 \label{s4pf3} \size_{f_{i}} ({\S^{1}}) \leq {1}, \, i=1,2, \qquad \|f_1\|_2^2 \leq \|f_1\|_p^{p} \|f_1\|^{2-p}_\infty \leq  \|f\|_p^p,\end{equation}
so that the first summand in \eqref{s4pf2} is bounded by invoking estimate \eqref{remarkexpest} with $A={1}$:
$$
\big|\{ |C_{{\S^{1}} } f_1 | \gtrsim \log_1(p')      \} \big| \leq \big|\{ |C_{{\S^{1}} } f_1 | \gtrsim    1    \} \big| \lesssim    {\|f_1\|_2^2}   \leq  \|f\|_p^p 
$$
We are only left with estimating the second summand in \eqref{s4pf2}. To do this, our plan is to apply \eqref{remarkexpest} again, once we have at hand the following multi-frequency projection lemma, which relies on the structure imposed on $\S_{\mathbb T}^{\cic{n}}$ by the lacunary sequence ${\cic{n}}$.  The first multi-frequency decomposition lemma of this sort  appeared in \cite{NOT} for the Fourier case, and modified Walsh versions of it have been successfully used in getting uniform estimates  \cite{OT} and endpoint bounds \cite{DD2}  for the quartile operator.   An argument along the same lines, but in the case of multiple \emph{lacunary} frequences, appears in \cite{DL}: our lemma is an $L^p$, $1<p<2$ reformulation of that argument.
\begin{lemma}\label{mflemma}
There is a function $g: \mathbb T \to \mathbb C$ with 
\begin{align} \label{laequality}
&\l f_2 , w_{s_1} \r  = \l g,w_{s_1} \r \qquad  \forall s \in \S^1,\\\label{lal2bd} & \|g\|_2^2 \lesssim (p')^2  |\{\M_p f> {1}\}|.
\end{align}
\end{lemma}
In view of \eqref{laequality} of Lemma \ref{mflemma}, we have that
$$
C_{{\S^{1}} } f_2= C_{{\S^{1}} } g, \qquad \size_{g} (\S^1) = \size_{f_2}(\S^1) \leq 1.
$$
Therefore, a further application of \eqref{remarkexpest} with $A={1}$, followed by \eqref{lal2bd},  yields
\begin{align*}
\big|\{ |C_{{\S^{1}} } f_2 | \gtrsim \log_1(p')     \} \big| & = \big|\{ |C_{{\S^{1}} } g| \gtrsim \log_1(p')     \} \big| \\ &\lesssim  \e^{-2\log_1 (p')}  {\|g\|_2^2}   \lesssim |\{\M_p f> {1}\}|,
\end{align*}
which once again has the correct measure by the maximal theorem. We have completed the proof of Theorem \ref{thmp}, up to showing Lemma \ref{mflemma}.
\vskip0.2cm\noindent \textit{Proof of Lemma \ref{mflemma}} \hskip0.1cm Let $I \in \cic{I}$  be the  maximal dyadic intervals of $\{\M_{p} f_1>1\}$;  for each $I\in \cic{I}$, let
$ t \in
T_I$ be the collection of all tiles  having $I_t=I$ and which are comparable under $\ll$ to some tile in $\{s_1: s \in \S^1\}$. These are obviously pairwise disjoint. 
The definition of $\S^1$ ensures that whenever $I_s \cap I$ for some $s \in \S_1$ and $I \in \cic{I}$, it must be that $I \subsetneq I_s$. It follows that  if $t \in T_I, s_1 \in \{s_1: s \in \S^1\}$ are related, then $t \ll s_1,s_2$. In particular, each $ t \in T_I$ must contain some lacunary frequency $n_j \in \cic{n}$; furthermore, by standard properties of Walsh wave packets, 
  $w_{s_1}$ (and $w_{s_2}$ as well, but we will not need this) is a scalar multiple of $w_t$ on $I$, and,  in particular, $
w_{s_1} \cic{1}_I $ belongs to $H_I$, the subspace of $L^2(I)$ spanned by $\{w_t: t \in T_I\}$. For functions $v \in  H_I$, one has the estimate
\begin{equation}
 \label{zyg}
 \|v\|_{L^{q}(I)}\lesssim  q \|v\|_{\mathrm{BMO}(I)}\lesssim_\theta q\|v\|_{L^2(I)}, \qquad 2<q<\infty;
\end{equation} the first bound is simply John-Nirenberg's inequality (and $\mathrm{BMO}(I)$ is the dyadic version), while the second is proved in \cite{Sagh}.
Since $\|f_2\|_{L^{p }(I)}=\|f \|_{L^{p }(I)} \leq 2$ by maximality of $I$ in $\{\M_p f>1\}$, it then follows that
$$
|(f_2 ,v)_{L^2( I)}|   \leq   \|f_2\|_{L^{p }(I)} \|v\|_{L^{p'}(I)} \lesssim_\theta   p'\|v\|_{L^2(I)}  \qquad \forall v \in H_I.
$$
Therefore $g_I $, the projection of $f_2 \cic{1}_I$ on $H_I$, satisfies  $\|g_I\|_{L^2(I)}  \lesssim p'$;  defining
$
g := \sum_{I \in \cic{I}} g_I, 
$
we see that
$$
\|g\|_{2}^2 = \sum_{I \in \cic{I}}|I| \|g_I\|_{L^2(I)}^2  \lesssim_\theta (p')^2\sum_{I \in \cic{I}} |I| = (p')^2  |\{\M_p f> {1}\}|,$$
that is, \eqref{lal2bd} holds.
Finally, in view of the above discussion, if $s_1 \in \{s_1: s \in \S^1\} $   
$$
\l f_2 , w_{s_1} \r = \sum_{I \in \cic{I}}\l f_2  , w_{s_1}\cic{1}_I  \r = \sum_{I \in \cic{I}}  \l f_2 \cic{1}_I  , c w_{t(s_1)} \r =   \sum_{I \in \cic{I}}  \l g_I  ,  w_{s_1} \r = \l g,w_{s_1} \r
$$ 
where $t(s_1)$ is the unique (if any) element $t$ of $T_I$ with $t \ll s_1$. This shows \eqref{laequality} and finishes the proof of the lemma.
  
\section{Remarks and complements} \label{sec5}
\subsection{A comparison with the argument in \cite{DL}} Therein,  estimate \eqref{estimateww} follows by upgrading the restricted weak-type version
 \eqref{rearrw},
  via Antonov's lemma \cite{ANT,SS} (which uses   the structure of the Walsh-Carleson kernel). In turn, \eqref{rearrw} is a consequence of  the restricted weak-type estimate
\begin{equation}
\l C^{\cic{n}}f, g \r \lesssim |F| \log_2\left(\frac{|G|}{|F|}\right)
\label{rwtest}
\end{equation}
for all sets $F,G\subset \mathbb{T}$, and all functions $|f|\leq \cic{1}_F$, $|g|\leq \cic{1}_{G'}$, with $G'$ being a suitably chosen  major subset of $G$. The proof of \eqref{rwtest} follows the usual Lacey-Thiele argument for boundedness of the unrestricted Carleson operator \cite{LT}; in particular, the dual quantity (density)
$$
\mathsf{dense}(\S) = \sup_{s \in  \S} \frac{|I_s \cap N^{-1}(\omega_s)\cap G |}{|I_s|}
$$
comes into play. For the unrestricted Carleson operator, the analogue of \eqref{rwtest} holds with a single logarithm; the improvement to double logarithm is possible thanks to a multifrequency  projection argument based on the same tools as Lemma \ref{mflemma} (in particular, an improvement over Hausdorff-Young inequality in the vein of \eqref{zyg}).

Our proof of Theorem \ref{mainthm} yields \eqref{estimateww} directly from the weak $L^p$ estimate
\begin{equation}
\label{finalp}B_{\cic{n}}(p):= \|C^{\cic{n}} f\|_{L^{p}(\mathbb{T}) \to L^{p,\infty}(\mathbb{T}) } \lesssim_\theta \log_1(p'), \qquad \forall \,1<p<2
\end{equation}  of Theorem \ref{thmp},   avoiding the need for extrapolation techniques. Moreover, our arguments do not employ density (which is also the key quantity in the  proof of the Fourier case \cite{LIE}), relying instead on the property that any collection of bitiles $\S \subset \S^{\cic{n}}_{\mathbb{T}}$ can be arranged into a forest $\F$ of trees with
\begin{equation} \label{suplac}
|C_\S f(x)| \lesssim_\theta \sup_{\T \in \F} |C_\T f(x)|,
\end{equation}
which exploits the lacunary structure, see Lemma \ref{fefftrick}. This property reflects the fact that the lacunary Carleson operator is essentially a supremum of (lacunarily) modulated Hilbert transforms acting on (essentially) pairwise disjoint regions of the time-frequency plane.     \subsection{Sharpness of Theorem \ref{thmp}} We conjecture that Theorem \ref{thmp}, summarized into \eqref{finalp}, is sharp in the following sense: for a generic  lacunary sequence,
$$
\limsup_{p \to 1^+}\frac{B_{\cic{n}}(p)}{\varphi(p')} =\infty \qquad \forall \varphi(t)= o(\log_1( t)), \; t \to \infty.
$$ 
We cannot quite prove this result; however, the weaker statement
$$
\limsup_{p \to 1^+}\frac{ B_{\cic{n}}(p)}{\varphi(p')} =\infty \qquad \forall \varphi(t)= o\left(\textstyle \frac{\log_1( t)}{\log_3(t)}\right), \; t \to \infty.
$$
must hold. If it were not so, an argument along the lines of the proof of Theorem \ref{mainthm} would contradict Konyagin's counterexample from \cite{K1}, that we have mentioned at the beginning of the paper. Similarly, proving that
$B_{\cic{n}}(p)\sim_\theta O(\log_1(p')/\log_3(p'))$ would allow  the removal of the 
quadruple-log term in Theorem \ref{mainthm}, thus yielding the sharp result. Our conjecture stems from deeming the  term $\log_3(p')$ as inconsequential, and  expresses the belief that  knowing the sharp weak $L^p$ constant would not suffice to prove the sharp analogue of Theorem \ref{mainthm}.
\subsection{Strong $L^1$ bounds} A further unresolved question concerns the largest Orlicz  space $X$ of functions $\mathbb T \to \C$ for which the bound
$$
\|\Ws^\star_{\cic{n}} f\|_{L^1(\mathbb{T})} \lesssim_\theta \|f\|_{X}
$$
holds. Since $\Ws^\star_{\cic{n}}$ is greater than each (discrete) $n_j$-modulated Hilbert transform, it follows that no Orlicz space $L_\varphi(\mathbb{T})$ with $$
\limsup_{t \to \infty}\frac{\varphi(t)}{t \log_1(t)} = 0
$$
embeds  into $X$. The (sharp, in terms of Orlicz norms) inclusion $L \log L(\mathbb{T}) \subset X$ is still unknown:
 the current best result \cite[(1.6) of Theorem 1.4]{DL} is that $L\log_1 L \log_2 L(\mathbb{T}) \subset X$. We can easily recover this result from Theorem \ref{thmp}: applying Marcienkiewicz interpolation, one turns the weak-type bound of Theorem \ref{thmp} into the strong bound
$$\|\Ws^\star_{\cic{n}}\|_{p \to p} \lesssim_\theta p' \log_1(p'), 
$$
which in turn implies $\Ws^\star_{\cic{n}}: L \log_1 L \log_2 L(\mathbb{T}) \to L^1(\mathbb{T})$, repeating the proof of the classical Yano extrapolation theorem. 

In relation to this, it is known that all sublinear translation invariant operators  of restricted weak type (1,1)  map $L \log_1 L  (\mathbb{T})$ into $L^1(\mathbb{T})$ (see for example \cite{Hag}). However, a result of Moon \cite{Moon} implies that  an operator of the form
$ 
Tf  = \sup_{n} |f * g_n | 
$ 
with each $g_n \in L^1(\mathbb{T})$, is of restricted weak type (1,1) if and only if it is of weak type (1,1). Since $\Ws^\star_{\cic{n}}$ is of this form, and it is not weak type (1,1), it cannot be restricted weak type (1,1) either. This suggests the need for direct methods in the search for a proof that $ \Ws^\star_{\cic{n}}$ is strong-type $L \log_1 L  (\mathbb{T}) \to L^1(\mathbb{T})$, possibly relying on \eqref{suplac}.

\subsection*{Acknowledgements}
The author wants to express his gratitude  to his Ph.\ D.\ thesis advisors Ciprian Demeter and Roger Temam for their hospitality during his April 2013 visit to the Institute of Scientific Computing and Applied Mathematics at Indiana University, where this article was finalized. The author also thanks    Victor Lie and Elena Prestini  for fruitful discussions on the subject of this paper and its presentation.

\bibliography{LacunaryFSDiPlinio}{}

\providecommand{\bysame}{\leavevmode\hbox to3em{\hrulefill}\thinspace}
\providecommand{\MR}{\relax\ifhmode\unskip\space\fi MR }
\providecommand{\MRhref}[2]{%
  \href{http://www.ams.org/mathscinet-getitem?mr=#1}{#2}
}
\providecommand{\href}[2]{#2}
\begin{thebibliography}{10}

\bibitem{ANT}
N.~Yu. Antonov, \emph{Convergence of {F}ourier series}, Proceedings of the {XX}
  {W}orkshop on {F}unction {T}heory ({M}oscow, 1995), vol.~2, 1996,
  pp.~187--196. \MR{1407066 (97h:42005)}

\bibitem{ADR}
J.~Arias-de Reyna, \emph{Pointwise convergence of {F}ourier series}, J. London
  Math. Soc. (2) \textbf{65} (2002), no.~1, 139--153. \MR{1875141
  (2002k:42009)}

\bibitem{CM}
Mar{\'{\i}}a~J. Carro and Joaquim Mart{\'{\i}}n, \emph{Endpoint estimates from
  restricted rearrangement inequalities}, Rev. Mat. Iberoamericana \textbf{20}
  (2004), no.~1, 131--150. \MR{2076775 (2005d:46153)}

\bibitem{CGMS}
Mar{\'{\i}}a~Jes{\'u}s Carro, Loukas Grafakos, Jos{\'e}~Mar{\'{\i}}a Martell,
  and Fernando Soria, \emph{Multilinear extrapolation and applications to the
  bilinear {H}ilbert transform}, J. Math. Anal. Appl. \textbf{357} (2009),
  no.~2, 479--497. \MR{2557660 (2010k:44008)}

\bibitem{D}
Ciprian Demeter, \emph{A guide to {C}arleson's theorem}, arXiv
  http://arxiv.org/abs/1210.0886.

\bibitem{DD2}
Ciprian Demeter and Francesco Di~Plinio, \emph{Endpoint bounds for the quartile
  operator}, J. Fourier Anal. Appl. \textbf{19} (2013), no.~4, 836--856.
  \MR{3089425}

\bibitem{DL}
Yen~Q. Do and Michael~T. Lacey, \emph{On the convergence of lacunary
  {W}alsh-{F}ourier series}, Bull. Lond. Math. Soc. \textbf{44} (2012), no.~2,
  241--254. \MR{2914604}

\bibitem{GMS}
Loukas Grafakos, Jos{\'e}~Mar{\'{\i}}a Martell, and Fernando Soria,
  \emph{Weighted norm inequalities for maximally modulated singular integral
  operators}, Math. Ann. \textbf{331} (2005), no.~2, 359--394. \MR{2115460
  (2005k:42037)}

\bibitem{Hag}
Paul~Alton Hagelstein, \emph{Problems in interpolation theory related to the
  almost everywhere convergence of {F}ourier series}, Topics in harmonic
  analysis and ergodic theory, Contemp. Math., vol. 444, Amer. Math. Soc.,
  Providence, RI, 2007, pp.~175--183. \MR{2423628 (2009i:42009)}

\bibitem{KAL}
N.~J. Kalton, \emph{Convexity, type and the three space problem}, Studia Math.
  \textbf{69} (1980/81), no.~3, 247--287. \MR{647141 (83m:46009)}

\bibitem{Sagh}
Elizabeth Kochneff, Yoram Sagher, and Ke~Cheng Zhou, \emph{B{MO} estimates for
  lacunary series}, Ark. Mat. \textbf{28} (1990), no.~2, 301--310. \MR{1084018
  (92j:42011)}

\bibitem{K1}
S.~V. Konyagin, \emph{Divergence everywhere of subsequences of partial sums of
  trigonometric {F}ourier series}, Proc. Steklov Inst. Math. (2005),
  no.~Function Theory, suppl. 2, S167--S175. \MR{2200228 (2006j:42007)}

\bibitem{K2}
Sergey~V. Konyagin, \emph{Almost everywhere convergence and divergence of
  {F}ourier series}, International {C}ongress of {M}athematicians. {V}ol. {II},
  Eur. Math. Soc., Z\"urich, 2006, pp.~1393--1403. \MR{2275651 (2008b:42006)}

\bibitem{LT}
Michael Lacey and Christoph Thiele, \emph{A proof of boundedness of the
  {C}arleson operator}, Math. Res. Lett. \textbf{7} (2000), no.~4, 361--370.
  \MR{1783613 (2001m:42009)}

\bibitem{LIE}
Victor Lie, \emph{On the pointwise convergence of the sequence of partial
  {F}ourier sums along lacunary subsequences}, J. Funct. Anal. \textbf{263}
  (2012), no.~11, 3391--3411. \MR{2984070}

\bibitem{Moon}
K.~H. Moon, \emph{On restricted weak type {$(1,\,1)$}}, Proc. Amer. Math. Soc.
  \textbf{42} (1974), 148--152. \MR{0341196 (49 \#5946)}

\bibitem{NOT}
Fedor Nazarov, Richard Oberlin, and Christoph Thiele, \emph{A
  {C}alder\'on-{Z}ygmund decomposition for multiple frequencies and an
  application to an extension of a lemma of {B}ourgain}, Math. Res. Lett.
  \textbf{17} (2010), no.~3, 529--545. \MR{2653686 (2011d:42047)}

\bibitem{OT}
Richard Oberlin and Christoph Thiele, \emph{New uniform bounds for a {W}alsh
  model of the bilinear {H}ilbert transform}, Indiana Univ. Math. J.
  \textbf{60} (2011), no.~5, 1693--1712. \MR{2997005}

\bibitem{SS}
Per Sj{\"o}lin and Fernando Soria, \emph{Remarks on a theorem by {N}. {Y}u.\
  {A}ntonov}, Studia Math. \textbf{158} (2003), no.~1, 79--97. \MR{2014553
  (2004i:42006)}

\bibitem{ThWp}
Christoph Thiele, \emph{Wave packet analysis}, CBMS Regional Conference Series
  in Mathematics, vol. 105, Published for the Conference Board of the
  Mathematical Sciences, Washington, DC, 2006. \MR{2199086 (2006m:42073)}

\end{thebibliography}
\bibliographystyle{amsplain}
\end{document}